\documentclass[11pt]{article}
\usepackage{amsthm}
\usepackage{amsmath,amssymb,graphicx,amsfonts,xspace}
\usepackage{psfrag,enumerate,pdfsync}
\usepackage[all,2cell,dvips]{xy} 
\newcommand{\comments}[1]{}

\usepackage{xcolor,framed}

\pdfoutput=1

\graphicspath{{./Figures/}}

\usepackage{amsfonts}
\usepackage{amssymb}
\usepackage[mathscr]{euscript}

\newtheorem{theorem}{\bf Theorem}

\newtheorem{definition}{Definition}
\newtheorem{lemma}{Lemma}

\newtheorem{proposition}{Proposition}
\newtheorem{remark}{Remark}

\newtheorem{example}{Example}

\def \CC {\mathbb{C}}

\def\ee{\begin{equation}}
\def\eee{\end{equation}}

\newcommand{\refe}[1]{(\ref{#1})}
\def \RR {\mathbb{R}}

\newcommand{\re}[1]{\mbox{Re}\left(#1\right)}
\newcommand{\im}[1]{\mbox{Im}\left(#1\right)}

\newcommand{\change}[1]{{#1}}

\newcommand{\spectrum}[1]{\mathcal{\sigma}\{#1\}}  
\def \system {\change{\mathscr{P}}}
\def \network {\mathscr{N}}

\def \graph {\mathscr{G}}
\newcommand {\graphset}[2] {\mathcal{G}_{#2}^{#1}}

\DeclareMathOperator{\Ker}{Ker}

\newcommand{\syncregion}[2]{\mathcal{S}_{#1}(#2)}
\newcommand{\intmatrix}[1]{L_{#1}}
\newcommand{\intmatrixred}[1]{\tilde{L}_{#1}}

\newcommand{\CCge}[1]{\CC_{\ge#1}}
\newcommand{\CCle}[1]{\CC_{\le#1}}
\newcommand{\CCg}[1]{\CC_{>#1}}
\newcommand{\CCl}[1]{\CC_{<#1}}

\DeclareMathOperator{\Span}{span}

\newcommand{\winding}[2]{\mbox{Ind}_{#1}\left(#2\right)}

\newcommand{\nyquistregion}[1]{\mathcal{N}_{#1}}

\begin{document}

\title{\bf \Large  
	Output-Feedback Synchronizability of Linear Time-Invariant Systems\footnote{This research was supported by the Natural Sciences and Engineering Research Council of Canada (NSERC).}
}

%
%
\author{Tian Xia and Luca Scardovi\footnote{Tian Xia and Luca Scardovi are with the Department of Electrical and Computer Engineering, University of Toronto, Toronto, ON, M5S 3G4, Canada;  e-mail: \texttt{t.xia@mail.utoronto.ca}; \texttt{scardovi@scg.utoronto.ca}.}}


%
%
%
%
%

\maketitle

\begin{abstract}
The paper studies the output-feedback synchronization problem for a network of identical, linear time-invariant systems. 
A criterion to test network synchronization is derived and the class of output-feedback synchronizable systems is introduced and characterized by sufficient and necessary conditions. In particular it is observed that output-feedback stabilizability is sufficient but not necessary for output-feedback synchronizability. 
In the special case of single-input single-output systems, conditions are derived in the frequency domain.  
The theory is illustrated with several examples.
\end{abstract}


\section{Introduction}

Synchronization has been recently a popular subject in the systems control community. This interest is motivated by the large array of phenomena exhibiting synchronization properties in physics and biology \cite{strogatz01}. 
Moreover, distributed problems arising in engineering applications, are commonly addressed in the context of synchronization theory \cite{fd-fb:13b, 6613520, 7040217,SE3}. 

We consider $N$ identical linear time-invariant (LTI) systems $\system=(A,B,C)$ 
\begin{equation} \begin{split} \label{eq:sysfunc} 
	\dot{x}_i &= A x_i + B u_i, \\
	y_i &= C x_i, \\	
\end{split} \end{equation}
where $x_i \in \mathbb{R}^n$, $u_i\in \mathbb{R}^m$, $y_i \in \mathbb{R}^q$,  $i = 1,\ldots,N$, $N>1$. The collection of systems \eqref{eq:sysfunc} is denoted by $\system^{N}$. The systems are coupled according to the following feedback 
\begin{equation} \begin{split} \label{eq:ctrsta}
	u_i &= K  \sum_{j=1}^{N} {\sigma_{i,j} (y_j - y_i)} , \quad i=1,\ldots,N,
\end{split} \end{equation}
where $K\in \RR^{m \times q}$.  
The problem of static output-feedback synchronization is to determine a matrix gain $K$ and an interconnection topology, defined by the coefficients $\sigma_{i,j} \in \RR$, such that the solutions of \refe{eq:sysfunc}, \refe{eq:ctrsta} asymptotically synchronize, i.e. $\lim_{t \rightarrow \infty} (x_i(t) - x_j(t)) = 0$ for every $i,j$ and every initial conditions. Both existence and design questions are of interest. In this paper we will address the existence question: determine under what conditions on \refe{eq:sysfunc}, a matrix $K$ and a communication topology $\sigma_{i,j}$ exist such that the solutions of \refe{eq:sysfunc}, \refe{eq:ctrsta}  synchronize. We will call this property {\em static output-feedback synchronizability} or, for short, {\em synchronizability}.   
The design problem is subject of ongoing research.

The output-feedback synchronization problem has been addressed in \cite{tuna08} by assuming that $B$ is the identity matrix and in \cite{meng13} by assuming that $C$ is the identity matrix. Both scenarios are particular cases of the general framework considered in this paper.  
In \cite{ma10} the synchronization problem is addressed by assuming that the columns of $B$ are contained in the image of $C^T$. Finally, a number of publications, see e.g., \cite{WeiRen:2008p184, WeiRen:2007je}, study synchronization for specific systems such as double integrators and harmonic oscillators. 

As for the output-feedback stabilization problem, the limitations imposed by static output-feedback can be overcome by using dynamic controllers.  
In \cite{scardovi09} and \cite{li10} it has been shown that, assuming that the interconnection topology satisfies a minimal connectivity requirement, stabilizability and detectability of the isolated systems is sufficient for the existence of a dynamic controller synchronizing the network. In \cite{scardovi09} the solution has been proposed in the case of time-varying communication topologies. Finally, \cite{Wieland} addressed the synchronization problem when the systems composing the network are not identical.

As shown in this paper, stabilizability and detectability are not sufficient for synchronizability. 
We first show that the synchronization problem can be addressed by studying the so called synchronization region (which depends on the structural properties of the uncoupled systems and the controller gain $K$) and the location of the eigenvalues of the interconnection matrix (which must be located inside the synchronization region in order for the network to synchronize). 
A connection between synchronizability and output-feedback stabilizability is established. It is shown that, somehow surprisingly, output-feedback stabilizability of the systems composing the network is a sufficient but not necessary condition for synchronizability. 
The notion of synchronization region and the synchronization criterion are then used to derive a graphical test to check synchronizability in networks of SISO systems. 

The paper is organized as follows. Section \ref{sec:not} introduces
the notation used throughout the paper and reviews preliminary material. Section \ref{sect:syncstactr} formalizes the synchronization problem. Section \ref{sect:syncstactr} and Section \ref{sect:siso_nyquist} present the main results of the paper. We conclude the paper by illustrating the theory with some examples and with some final remarks. Preliminary results related to this paper appeared in \cite{MTNS}.

\section{Preliminaries} \label{sec:not}
\subsection{Notations}
The following notations will be used throughout the paper. 
We denote the open right (left) half complex plane by $\CCg{0}$ ($\CCl{0}$), \change{and} the closed right (left) half complex plane by $\CCge{0}$ ($\CCle{0}$).
We denote by $\mathbf{1}_n$ the column vector in $\CC^n$ containing $1$ in each entry.  
Given a complex matrix $M \in \CC^{n \times m}$, $M^T$ denotes its transpose and $M^{*}$ its conjugate transpose. Given a square matrix $M \in \CC^{n \times n}$, $\spectrum{M}$ denotes its spectrum (defined as the multiset of the eigenvalues of $M$). \change{The matrix $M$ is called Hurwitz if $\spectrum{M} \subseteq \CCl{0}$.} We write $M>0$ ($M\geq0)$ to indicate that $M$ is positive-definite (positive-semidefinite). 
The identity matrix in $\CC^{n \times n}$ is denoted by $I_n$. 

\subsection{Graph theory}

A \emph{directed graph} $\graph$ consists of the triple $\left(\mathcal{V}, \mathcal{E}, \Sigma \right)$, where $\mathcal{V} = \left\{1, 2, \ldots, N \right\}$ is the set of nodes, $\mathcal{E} \subseteq \mathcal{V}\times\mathcal{V}$ is the set of edges and $\Sigma \in \mathbb{R}^{N\times N}$ is a weighted adjacency matrix. Each element $\sigma_{i,j}$ (an element of $\Sigma$) is nonzero if and only if $\left( i,j \right) \in \mathcal{E}$. When $(i,j) \in \mathcal{E}$, node $j$ is called a {\em neighbor} of node $i$. We assume that there are no self-loops and therefore $\sigma_{i,i} = 0$ for $i = 1 \dots N$. Unless differently stated, we allow for negative weights $\sigma_{i,j}$. The set of graphs with the properties above is denoted by $\graphset{N}{}$. Two subsets of $\graphset{N}{}$ are given special notations: $\graphset{N}{+}$ is the subset of graphs with non-negative weights ($\sigma_{i,j} \geq 0$); while $\graphset{N}{u}$ is the subset of graphs characterized by symmetric matrices $\Sigma$. Given a graph $\graph \in \graphset{N}{+}$, a path between two nodes $n_{1},n_{l}$ is a sequence of nodes $\{n_{1},n_{2},\ldots,n_{l}\}$ such that $n_{i},n_{i+1}$ is an edge for $i=1,\ldots l-1$. A node $n_{b}$ is called \emph{reachable} from a node $n_{a}$ if there exists a path between $n_{a}$ and $n_{b}$. A node is {\em globally reachable} if it is reachable from every other node. 
Given a graph $\graph \in \graphset{N}{}$, we define the {\em interconnection matrix} $\intmatrix{\graph}$ as the $N\times N$ matrix with elements
\begin{equation} \begin{split} \label{eq:intmatrix}
	[\intmatrix{\graph}]_{i,j} &:= 
	\begin{cases}
		\sum_{k=1}^{N}{\sigma_{i,k}}, & i=j, \\
		-\sigma_{i,j}, & i \neq j.
	\end{cases}
\end{split} \end{equation}
The matrix $\intmatrix{\graph}$ always contains $0$ and $\mathbf{1}_N$ as an eigenvalue-eigenvector pair (since $\intmatrix{\graph}$ has zero row sum). $\intmatrix{\graph}$ has special properties when the graph $\graph$ belongs to $\graphset{N}{+}$ or $\graphset{N}{u}$.
For graphs in $\graphset{N}{u}$, $\intmatrix{\graph}$ is a symmetric matrix and has therefore real eigenvalues. 
For graphs in  $\graphset{N}{+}$, the associated interconnection matrix $\intmatrix{\graph}$ is called {\em Laplacian matrix}, and it is the generalization of the standard Laplacian matrix defined for undirected graphs (see e.g., \cite{agaev05} and references therein). 
All the eigenvalues of a Laplacian matrix have non-negative real part and the (always present) zero eigenvalue has multiplicity one if and only if the graph \change{contains a globally reachable node} \cite{WeiRen:2005wm}.  
Let $Q_N$ be a matrix belonging to $\RR^{(N-1) \times N}$ and satisfying the following properties
\begin{equation} \label{eq:prop}
	Q_N \mathbf{1}_N = 0,\quad Q_N^T Q_N = \Pi_N,\quad Q_N Q_N^T = I_{N-1}, \\
\end{equation}
where $\Pi_N := I_N - \frac{1}{N} \mathbf{1}_N \mathbf{1}_N^T$ is the \change{projector onto} the subspace orthogonal to $\Span(\mathbf{1}_N)$. Given a graph $\graph$, the \emph{reduced interconnection matrix} is defined by
\begin{equation} \begin{split} \label{eq:intmatrixred}
	\intmatrixred{\graph} &:= Q_N \intmatrix{\graph} Q_N^T. \\
\end{split} \end{equation}
where $Q_N \in \RR^{(N-1) \times N}$ and satisfies the properties \refe{eq:prop}. The spectrum of $\tilde L_{\graph}$ is the spectrum of $L_{\graph}$ with one instance of the zero eigenvalue removed, i.e. $\spectrum{\intmatrixred{\graph}} = \spectrum{\intmatrix{\graph}} \setminus \{0\}$  \cite{young}. Therefore, when $\graph \in \graphset{N}{+}$ \change{contains a globally reachable node}, $\spectrum{\intmatrixred{\graph}} \subset \CCg{0}$.
These properties are invariant to the choice of $Q_N$  \cite{young}.

\section{Synchronization criterion and synchronizability} \label{sect:syncstactr} 

We represent the network coupling structure with a directed graph. For this purpose we introduce $N$ nodes labeled consecutively from $1$ to $N$. Each node represents a system in the network. If a coefficient $\sigma_{i,j}=0$ then the edge connecting node $i$ to node $j$ is not present. If $\sigma_{i,j}\neq 0$ the relative edge exists and its weight is determined by the (possibly negative) coefficient ${\sigma_{i,j}}$. We call the resulting graph $\graph$ the \emph{communication topology}. A collection of systems \refe{eq:sysfunc} together with a feedback matrix $K\in \RR^{m \times q}$ and a communication topology $\graph$ form a {\em network} that will be denoted by $\network:=(\system^N, K, \graph)$.   
The next definition formalizes the notion of network synchronization.
\begin{definition}
	A network $\network = (\system^{N}, K, \graph)$, is said to \emph{synchronize} if 
	\[
	\lim_{t \rightarrow \infty} (x_i(t) - x_j(t)) = 0, 
	\]
for $i,j = 1, 2, \dots, N$ and for all initial conditions. 
\end{definition}

Network synchronization depends on the structural properties of the system $\system$, on the graph $\graph$ and on the choice of the matrix $K\in \RR^{m \times q}$. In this paper we investigate the structural properties of $\system$ such that $(\system^N, K, \graph)$ synchronizes for some $K$ and $\graph$.
%
%


\begin{definition}[Synchronizability]
A collection of systems $\system^{N}$ is \emph{output-feedback synchronizable} (OFS) if there exist a matrix $K\in \RR^{m \times q}$ and a graph $\graph \in \graphset{N}{}$ such that the network $(\system^{N}, K, \graph)$ synchronizes.
\end{definition}





\noindent We will make use of the following notion of \emph{synchronization region}. 

\begin{definition} \label{defn:syncregion_sta}
	Given a system $\system=(A,B,C)$ and a matrix $K\in \RR^{m \times q}$, the \emph{synchronization region} $\syncregion{\system}{K}$ is the subset of the complex plane defined by
	\begin{equation} \label{eq:sync}\begin{split}
		\syncregion{\system}{K} := \{ s \in \mathbb{C} \bigm| A - sBKC \mbox{ is Hurwitz} \}. \\
	\end{split} \end{equation}
\end{definition}
The term synchronization region is justified by the synchronization criterion presented below. 

\begin{theorem} \label{thm:synccritsta}
A network $\network=(\system^{N}, K, \graph)$ synchronizes if and only if $\spectrum{\intmatrixred{\graph}} \subseteq \syncregion{\system}{K}$. 
	
\end{theorem}

\begin{proof}
Define $x = [x_1^T, \ldots, x_N^T]^T$ and rewrite \refe{eq:sysfunc}, \refe{eq:ctrsta} in compact form as
\begin{equation*} \begin{split} 
	\dot{x} &= (I_N \otimes A - \intmatrix{\graph} \otimes BKC) x.
\end{split} \end{equation*}
Let   
$\mathcal{X}_{\parallel}:=\{x\in \RR^{nN} \bigm| \left(\Pi_{N} \otimes I_{n}\right) x = 0\}$ be the synchronization subspace and $\mathcal{X}_{\perp}:=\{x\in \RR^{nN} \bigm| \left(\frac1 N 1_{N}1_{N}^{T}\otimes I_{n}\right)x = 0\}$ its orthogonal complement, called the transversal subspace. The network synchronizes if and only if, for any initial conditions, the projection of the state \change{onto} the transversal subspace converges to zero asymptotically.   

Let $Q_{N}$ be a $(N-1)\times N$ real matrix satisfying \eqref{eq:prop} and define the new set of coordinates $x_{\perp} := (Q_N \otimes I_n) x$ and $x_{\parallel} := \frac{1}{N} (\mathbf{1}_N^T \otimes I_n) x$. Notice that $(Q_N \otimes I_n)$ is a partial isometry. In fact, it is an isometry on the orthogonal complement of its kernel (the transversal subspace), as $\|(Q_N \otimes I_n)x\| = \|{x}\|$ when $x \in \mathcal{X}_{\perp}$. This implies that synchronization is equivalent to asymptotic stability of the origin of     
\begin{equation*} \begin{split} 
	\dot{x}_{\perp} &= (I_{N-1} \otimes A - \intmatrixred{\graph} \otimes BKC) x_{\perp}, 
\end{split} \end{equation*}
which, \change{in turn}, is equivalent to
$(I_{N-1} \otimes A - \intmatrixred{\graph} \otimes BKC)$ being Hurwitz. 

\change{By using the Jordan decomposition, we can write $P^{-1} \intmatrixred{\graph} P = \Lambda$, where $P \in \CC^{n \times n}$ is an invertible matrix and $\Lambda$ is an upper-triangular matrix with the same eigenvalues of $\intmatrixred{\graph}$ as its diagonal components. By using the properties of the Kronecker product and the Jordan decomposition of $\intmatrixred{\graph}$  we obtain
\begin{align*}
	(I_{N-1} \otimes A - \intmatrixred{\graph} \otimes BKC) = (P \otimes I_n)(I_{N-1} \otimes A - \Lambda \otimes BKC){(P \otimes I_n)}^{-1}.
\end{align*}
Notice that the matrix $(I_{N-1} \otimes A - \Lambda \otimes BKC)$ is complex and block upper-triangular with one diagonal block $A - \lambda_i BKC$ for each (possibly complex) eigenvalue $\lambda_i$, $i=1,2,\ldots,n-1$, of $\intmatrixred{\graph}$. Therefore, $(I_{N-1} \otimes A - \intmatrixred{\graph} \otimes BKC)$ is Hurwitz if and only if all diagonal blocks $A - \lambda_i BKC$ are Hurwitz.}
This is equivalent to the condition $\spectrum{\intmatrixred{\graph}} \subseteq \syncregion{\system}{K}$.
\end{proof}

Some comments on the synchronization region $\syncregion{\system}{K}$ are now in place. It is clear from \refe{eq:sync} that the synchronization region depends only on system $\system$ and the matrix gain $K$. Therefore it does not depend from the interconnection topology. The synchronization region is an open set and, since the eigenvalues of $A - sBKC$ and $A - s^* BKC$ are complex conjugated, it is symmetric with respect to the real axis. 

According to Theorem \ref{thm:synccritsta}, once the gain $K$ has been fixed, the synchronization region defines the subset of the complex plane where the eigenvalues of the interconnection matrix must be located in order for the network to synchronize. The synchronization region, therefore, provides information about the interconnection topologies required to achieve synchronization. For example, if the synchronization region does not intersect the real axis, the communication topology must necessarily be a non symmetric directed graph in order for the network to synchronize\footnote{This follows from the fact that the interconnection matrix of an undirected graph has all the eigenvalues on the real axis.}.  


The choice of not restricting the communication topology to the graphs with non negative weights $\graphset{N}{+}$, is justified by the next example, where it is shown that there are collections of systems that cannot be synchronized unless the interconnection topology contains both positive and negative weights.


\begin{example} \label{ex:graph_posneg}
	Consider the system $\system = (A,B,C)$ where 
	\begin{equation*} 
		A =
		\begin{bmatrix}
			0		& 1		& 0 	& 0		\\
			0		& 0		& 1		& 0		\\
			0		& 0		& 0		& 1		\\
			1		& 2		& 0		& -2	\\
		\end{bmatrix} \hspace{-1.5mm},
		 \quad
		B =
		\begin{bmatrix}
			0		\\
			0		\\
			0		\\
			1		\\
		\end{bmatrix}\hspace{-1.5mm}, \quad 
		C = 
		\begin{bmatrix}
			0		& 1		& 1.5	& 1		\\
		\end{bmatrix}\hspace{-1.5mm}.
	\end{equation*}
	The synchronization region associated to the gain $K = 1$ is illustrated  in Fig. \ref{fig:ex2syncregion} (shaded region). According to Proposition 1 in \cite{agaev05}, the spectrum of any $N$-dimensional Laplacian matrix is contained in the set
	\[
	\left\{ s= \delta+i\omega \bigm| \| \omega \| \leq \| \delta \| \cot \frac{\pi}{N}, \; \delta \geq 0 \right\}. 
	\]
	In the case of $N=3$ this set corresponds to the dashed region in Fig. \ref{fig:ex2syncregion}.
	Since $\spectrum{\intmatrixred{\graph}}$ is disjoint from $\syncregion{\system}{1}$, any network $(\system^{3}, 1, \graph)$ does not achieve synchronization whenever the graph $\graph \in \graphset{3}{+}$. Moreover, since the system is SISO, the same conclusion holds for any choice of the control gain.
	However, as we will show later (see Lemma \ref{lem:stasyncable}), the collection of systems is output-feedback synchronizable and, therefore, there exist a gain $K$ and a graph $\graph \in \graphset{3}{}$ such that the resulting network synchronizes. 
	 
	\begin{figure}
		\centering
		\includegraphics[width=6cm]{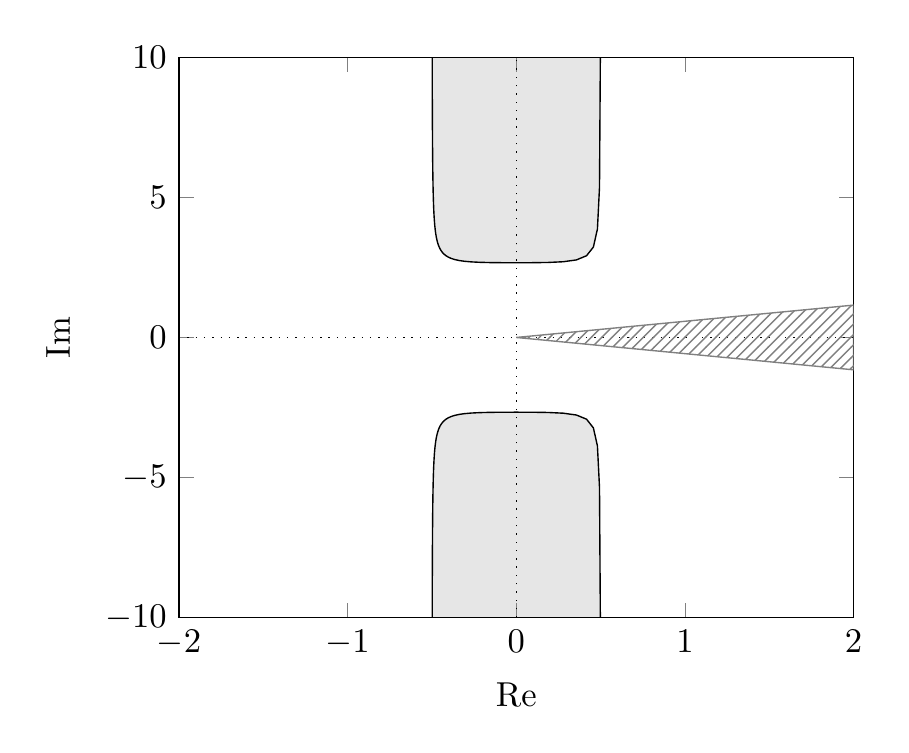}
		\caption{Synchronization region and eigenvalues region for any graph in $\graphset{3}{+}$ (Example \ref{ex:graph_posneg})}
		\label{fig:ex2syncregion}
		\vspace {-0.2 cm}
	\end{figure}
\end{example}




According to Theorem \ref{thm:synccritsta}, the existence of a non-empty synchronization region is a necessary condition for \change{output-feedback} synchronizability.
The next result further characterizes the relationship between \change{output-feedback} synchronizability and the properties of the systems composing the network. It is shown that i) stabilizability and detectability of $\system$ are necessary conditions for \change{output-feedback} synchronizability of $\system^{N}$ for any $N$; ii) the synchronization region (and therefore \change{output-feedback} synchronizability) depends only on the controllable and observable \change{subsystem} of $\system$. 

\begin{theorem} \label{thm:ctrobs_subsys}
Let $\tilde{\system}$ be the controllable and observable \change{subsystem} of $\system$. The synchronization region  
	\begin{equation*}
		\syncregion{\system}{K} =
		\begin{cases}
			\syncregion{\tilde{\system}}{K}, & \mbox{if } \system \mbox{ is stabilizable and detectable,} \\
			\emptyset , & \mbox{otherwise}.
		\end{cases}
	\end{equation*}
\end{theorem}
\begin{proof}
	By expressing $\system = (A, B, C)$ in the canonical Kalman form, the matrix $A - sBKC$ can be written as 
		\begin{equation*}
			\begin{bmatrix}
			A_{1,1}							& A_{1,2} - sB_1KC_2	& A_{1,3}							& A_{1,4} - sB_1KC_4	\\
			0										& A_{2,2} - sB_2KC_2	& A_{2,3}							& A_{2,4} - sB_2KC_4	\\
			0										& 0										& A_{3,3}							& A_{3,4}							\\
			0										& 0										& 0										& A_{4,4}							\\
			\end{bmatrix},
		\end{equation*}
	where $A_{1,1}$ is the controllable/unobservable part, $A_{2,2}$ is the controllable/observable part, $A_{3,3}$ the uncontrollable/unobservable part, $A_{4,4}$ the uncontrollable/observable part, and $B =[B_1^T, B_2^T, 0, 0]^T$.
	\change{If} $\system$ is stabilizable and detectable, $A_{1,1}$, $A_{3,3}$, and $A_{4,4}$ are Hurwitz and $\syncregion{\system}{K} = \{s\in\CC \bigm| \spectrum{A_{2,2} - sB_2 K C_2} \subseteq \CCl{0}\} = \syncregion{\tilde{\system}}{K}$. Otherwise, $\syncregion{\system}{K} = \emptyset$.
\end{proof}

The next result is a first characterization of output-feedback synchronizability.

\begin{lemma} \label{lem:stasyncable}
Given a LTI system $\system$, the following facts hold true
	\begin{enumerate}
		\item[i)] If $N$ is even, $\system^{N}$ is OFS if and only if $\left(\syncregion{\system}{K} \cap \RR\right) \neq \emptyset$ for some $K$.
		\item[ii)] If $N$ is odd, $\system^{N}$ is OFS if and only if $\syncregion{\system}{K} \neq \emptyset$ for some $K$.
	\end{enumerate}
\end{lemma}
\begin{proof}
	i): (Sufficiency) By assumption there exists a matrix $K$ such that $\left(\syncregion{\system}{K} \cap \RR\right) \neq \emptyset$.
	Let $p \in \left(\syncregion{\system}{K} \cap \RR\right)$. The graph $\graph$ with weights $\sigma_{i,j} = p/N$, $i,j =1,\ldots,N$, satisfies $\spectrum{\intmatrixred{\graph}} = \{p, \dots, p\} \subseteq \syncregion{\system}{K}$ and therefore $\system^{N}$ is \change{OFS}. \\
	(Necessity)
	Since $\system^N$ is \change{OFS}, there exist $K$ and $\graph \in \graphset{N}{}$ such that $(\system^N, K, \graph)$ synchronizes and, by Theorem \ref{thm:synccritsta}, $\spectrum{\intmatrixred{\graph}} \subseteq \syncregion{\system}{K}$. Since $N$ is an even natural number, the reduced interconnection matrix $\intmatrixred{\graph}$ has at least one eigenvalue on the real axis. Since $\spectrum{\intmatrixred{\graph}} \subseteq \syncregion{\system}{K}$ we conclude that $\syncregion{\system}{K} \cap \RR \neq \emptyset$.\\
	ii): (Sufficiency)
	Let $p \in \syncregion{\system}{K}$. By symmetry of $\syncregion{\system}{K}$ with respect to the real axis, $p^* \in \syncregion{\system}{K}$. In order to show that the network is \change{OFS}, it is sufficient to show that there exists a graph $\graph \in \graphset{N}{}$, $N$ odd,  associated with an interconnection matrix $\intmatrix{\graph}$ with spectrum $\spectrum{\intmatrix{\graph}} = \{0,\underbrace{p, p^*,\ldots,p, p^*}_{N-1}\}$.     
	Define $L$ as
	\begin{equation*} \begin{aligned}
		L &= P
		\begin{bmatrix}
			0	& ~		& ~ 	& 0		\\
			~	& R		& ~		& ~		\\
			~	& ~		& \ddots	&	~	\\
			0	& ~		& ~ 	& R
		\end{bmatrix}
		P^{-1},& 
		R &=
		\begin{bmatrix}
			\re{p}		& -\im{p}	\\
			\im{p}		& \re{p}	
		\end{bmatrix}, \\
	\end{aligned} \end{equation*}
	where $P$ is any invertible real matrix for which $\mathbf{1}_N$ is the first column. Since $L \mathbf{1}_N = 0$, by using \refe{eq:intmatrix}, we can choose the adjacency matrix coefficients as  $\sigma_{i,j} = -l_{i,j}$, for any $i\neq j$, and $\sigma_{i,i}=0$ for any $i$. This defines a graph $\graph$ and the interconnection matrix $L = \intmatrix{\graph}$. We conclude that, by construction, $\spectrum{\intmatrixred{\graph}} = \{\underbrace{p, p^*,\ldots,p, p^*}_{N-1}\} \subseteq \syncregion{\system}{K}$ and therefore $\system^{N}$ is \change{OFS}.\\
	(Necessity) Follows the same lines of the proof of i) (necessity). 	
\end{proof}

Lemma \ref{lem:stasyncable} can be used to relate \change{output-feedback} synchronizability of $\system^{N}$ to the structural properties of $\system$. 
In the next result we show that, in the case of even collections, \change{output-feedback} synchronizability of $\system^{N}$ is equivalent to output-feedback stabilizability\footnote{A system $\system = (A, B, C)$ is output-feedback stabilizable if there exists $K$ such that $A - BKC$ is Hurwitz} of $\system$. We remark here that finding a characterization of output-feedback stabilizability is an open problem \cite{Syrmos:1997hz}. 


\begin{theorem} \label{thm:stasyncableeven}
If $N$ is even, $\system^{N}$ is OFS if and only if the system $\system$ is output-feedback stabilizable.
\end{theorem}

\begin{proof}
	(Sufficiency)
	Output-feedback stabilizability of $\system$ guarantees the existence of real matrix $K$ such that $A - BKC$ is Hurwitz. Hence, $\{1\} \subset \syncregion{\system}{K}$ and, by Lemma \ref{lem:stasyncable}, $\system^{N}$  is OFS.
	
	(Necessity)
	Since $N$ is even and  $\system^{N}$ is \change{OFS}, by Lemma \ref{lem:stasyncable}, there exists $K$ such that $\left(\syncregion{\system}{K} \cap \RR\right) \neq \emptyset$. Therefore, there exists a real number $p$ such that $A - pBKC$ is Hurwitz. We conclude that $\system$ is output-feedback stabilizable with feedback matrix $pK$.
\end{proof}


Notice that in the first part of the proof of Theorem \ref{thm:stasyncableeven} we do not use the assumption that $N$ is even, therefore, if $\system$ is output-feedback stabilizable, $\system^{N}$ is output-feedback synchronizable for all $N$. \change{The converse also holds as a special case of Theorem \ref{thm:stasyncableeven}.} This observation is summarized in the following \change{Proposition}.
\begin{proposition} \label{corol1}
$\system^{N}$ is OFS for all $N$ if and only if $\system$ is output-feedback stabilizable.
\end{proposition}
It turns out, as illustrated in the following example, that output-feedback stabilizability is not necessary to achieve synchronization if $N$ is odd.
\begin{example}\label{ex:odd}
	Consider the system $\system = (A,B,C)$ defined by
	\[
		A =
		\begin{bmatrix}
		0			& 1			& 0			& 0			\\
		0			& 0			& 1			& 0			\\
		0			& 0			& 0			& 1			\\
		1			& 11		& 9			& 8			\\
		\end{bmatrix}, \quad
		B =
		\begin{bmatrix}
		0			\\
		0			\\
		0			\\
		1			\\
		\end{bmatrix}, \quad
		C =
		\begin{bmatrix}
		0			& 6			& 6			& 6			\\
		\end{bmatrix}. 
	\]
	
The characteristic polynomial of $A - BKC$ is $\chi_{A - BKC} (\lambda) = \lambda^4 + (6K - 8) \lambda^3 + (6K - 9) \lambda^2 + (6K - 11) \lambda - 1$. By the Routh-Hurwitz criterion, $A - BKC$ is non Hurwitz regardless of the choice of $K \in \RR$ and therefore $\system$ is not output-feedback stabilizable. By Theorem \ref{thm:stasyncableeven}, $\system^{N}$ is not OFS if $N$ is even. However, setting $K = 1$ results in a non-empty synchronization region $\syncregion{\system}{1}$ (see Fig. \ref{fig:ex3syncregion}). From Lemma \ref{lem:stasyncable} we conclude that $\system^{N}$ is OFS if $N$ is odd. As an example, we can fix $N=3$, $K = 1$ and the graph $\graph$ shown in Fig. \ref{fig:ex3syncregion}. The spectrum of the reduced interconnection matrix associated to the graph $\graph$ is $\spectrum{\intmatrixred{\graph}} = \{\left(3 \pm i\sqrt{3}\right)/2\} \subset \syncregion{\system}{1}$ and, from Theorem \ref{thm:synccritsta}, we conclude that the network $\network =(\system^{3},1,\graph)$ synchronizes. 
\end{example}

\begin{figure}[t]
	\centering
	\includegraphics[width=5.5cm]{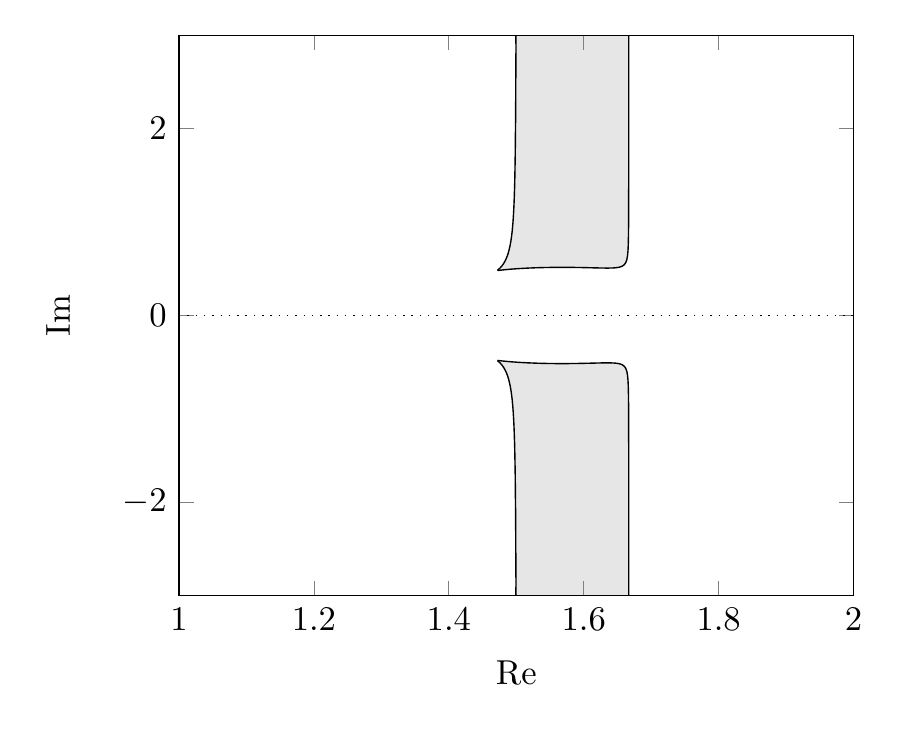}
 \quad \quad\includegraphics[width=5 cm]{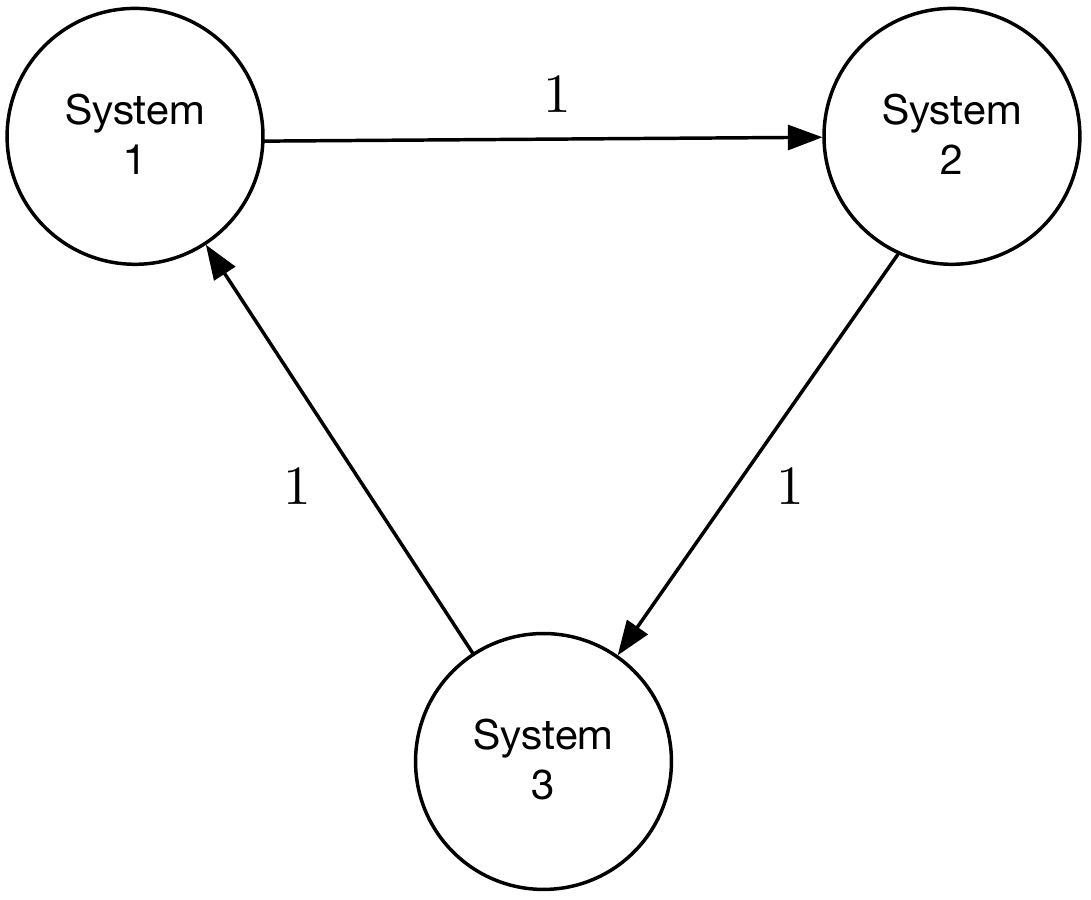}
	\caption{Left: Synchronization region of Example \ref{ex:odd} ($K =1$). Right: Circulant interconnection topology used in Example \ref{ex:odd}.
	} 
	\label{fig:ex3syncregion}
	\vspace {-0.2 cm}
\end{figure}

The main theorem in \cite{kucera95} provides a sufficient and necessary condition for output-feedback stabilizability of $\system$. 
In the next Theorem we generalize those conditions to obtain a sufficient and necessary condition for \change{output-feedback} synchronizability of $\system^N$, when $N$ is odd. We first present a lemma required in the proof. 

\begin{lemma} \label{lemma:hurwitz_alt}
	Let $A, H \in \CC^{n \times n}$ where $H$ is positive-semidefinite and $\Ker H$ does not contain any eigenvector of $A$ corresponding to an eigenvalue in $\CCge{0}$. The matrix $A$ is Hurwitz if and only if there exists positive-semidefinite $P\in \CC^{n \times n}$ such that $A^*P + PA = -H$.
	If $A$ and $H$ are real matrices, then $P$ can be chosen to be real.
\end{lemma}

\begin{proof}
	\change{
	(Necessity)
	Let 
	\[
	P := \int_{0}^{\infty} e^{A^*t} H e^{At} dt.
	\]
	$P$ is well-defined because the integrand is absolutely integrable. Since $H$ is positive-semidefinite $P$ is also positive-semidefinite and
	\begin{equation*}
		\frac{d}{dt} e^{A^*t} H e^{At} = A^* \left(e^{A^*t} H e^{At}\right) + \left(e^{A^*t} H e^{At} \right)A.
	\end{equation*}
	Integrating both sides from zero to infinity yields $-H = A^* P + PA$.
	
	(Sufficiency)
	We proceed by contradiction. Assume that $A$ is not Hurwitz, then there exists eigenvalue $\lambda \in \CCge{0}$ and eigenvector $v$ such that $Av = \lambda v$. Hence, $v^*Hv = -v^*(A^*P+PA)v = -2 \re{\lambda} v^*Pv \leq 0$. Positive-semidefiniteness of $H$ implies that $v^*Hv = 0$. Define $\phi(t) = (v + tHv)^* H (v + tHv)$ for $t \in \RR$ and notice that $\phi(0) = 0$. By positive-semidefiniteness of $H$, $\phi(t) \geq 0$. Therefore, the derivative $\phi'(0) = 2(Hv)^*(Hv) = 0$. We conclude that $v \in \Ker H$, which contradicts the hypothesis.
	}  
\end{proof}

\begin{theorem} \label{thm:stasyncableodd}
If $N$ is odd, $\system^{N}$ is OFS if and only if
	\begin{enumerate}
	\item[i)] $(A,B)$ is stabilizable;
	\item[ii)] $(A,C)$ is detectable;
	\item[iii)] There exist real matrix $K$, complex number $s$, and Hermitian positive-semidefinite $P$ such that 
		\begin{equation} \label{eq:lyap}
			\left (A - sBKC\right)^* P + P \left(A - sBKC\right) + C^T C + C^T K^T KC = 0. 
		\end{equation}
	\end{enumerate}
\end{theorem}
\begin{proof}
	(Sufficiency)
	Choose $K$, $P$, and $s$ such that  \refe{eq:lyap} is satisfied. Define $H := C^T C + C^TK^T KC$ and rewrite \refe{eq:lyap} as
	\begin{equation} \label{eq:lyap2}
	\left (A - sBKC\right)^* P + P \left(A - sBKC\right) = -H. 
	\end{equation}
	Since $(A, C)$ is detectable and detectability is unaffected by output-feedback, the kernel of $H$ does not contain any eigenvector associated to an eigenvalue with nonnegative real part of $A-sBKC$. Therefore $A-sBKC$ and $H$ satisfy the assumptions of Lemma \ref{lemma:hurwitz_alt}. Since $P$ in \refe{eq:lyap2} is positive-semidefinite, by Lemma \ref{lemma:hurwitz_alt}, the matrix $A-sBKC$ is Hurwitz. By Lemma \ref{lem:stasyncable} we conclude that the collection is OFS. 

	(Necessity)
	Since the collection is \change{OFS}, Lemma \ref{lem:stasyncable} guarantees the existence of $K$ such that $\syncregion{\system}{K} \neq \emptyset$.	By Theorem \ref{thm:ctrobs_subsys}, $\system$ is stabilizable and detectable. It remains to prove iii). Choose $s \in \CC$ such that $A - sBKC$ is Hurwitz. Define the positive-semidefinite and Hermitian matrix $H = C^T C + C^T K^T KC$. Then, from a standard Lyapunov argument, there exists a Hermitian positive-semidefinite $P$ such that	$(A-sBKC)^*P + P(A-sBKC) + H = 0$.
\end{proof}

\change{Theorem \ref{thm:stasyncableeven} and Theorem \ref{thm:stasyncableodd} provide a complete characterization of output-feedback synchronizability for even and odd collections respectively. We end this section with a final remark.
\begin{remark}\label{rem:state1}
	In the case of state feedback, i.e. when  $C = I_n$, \change{output-feedback} synchronizability does not depend on the number of systems in the collection and it is equivalent to stabilizability. This follows from the following simple argument.
	Stabilizability of $\system$ implies the existence of state feedback $K$ such that $A - BK$ is Hurwitz. Therefore, from \change{Proposition \ref{corol1}}, $\system^N$ is \change{OFS} for all $N$. Stabilizability is also necessary by Theorem \ref{thm:ctrobs_subsys}.  
\end{remark}}

\section{SISO Systems} \label{sect:siso_nyquist}

In this section we particularize our results to networks of SISO systems. 
Given a SISO system $\system = (A,b,c)$, its transfer function
is $H(s) = c(sI - A)^{-1}b$. In assuming that there is no throughput we have made the restriction to strictly proper transfer functions (the relative degree of $H$ is at least one). 

The Nyquist contour $\gamma : [-\infty, \infty] \rightarrow \CC$ is the oriented curve defined by $\gamma(\omega) = H(i\omega)$\footnote{The definition of $\gamma$ extends to $\pm \infty$ by continuity (i.e. $\gamma(\pm \infty) = 0$ when $H(s)$ is a strictly proper rational function)}. 
The winding number\footnote{The winding number of an oriented curve around a point is the number of counterclockwise rotations of the curve around the point.} of $\gamma$ around a point $s \in \CC$ is denoted $\winding{\gamma}{s}$. 
As routinely done in frequency domain analysis, if  $s_i=i\omega_i$ are poles of $H(s)$ on the imaginary axis we define
\begin{equation*} 
	\gamma_\epsilon(\omega) =
	\begin{cases}
		H(i\omega_i + \epsilon e^{\frac{i\pi}{2\epsilon}(\omega - \omega_i)}), & \|\omega - \omega_i\| < \epsilon, \\
		H(i\omega), & \mbox{otherwise},
	\end{cases}
\end{equation*}
and 
\begin{equation*} 
	\winding{\gamma}{s} := 
	\lim_{\epsilon \rightarrow 0^+} \winding{\gamma_\epsilon}{s}.
\end{equation*}

\begin{definition} \label{defn:nyquistregion}
	Given a system $\system$ with transfer function $H(s)$ and Nyquist contour $\gamma$, the \emph{stable Nyquist region} is
	\begin{equation*}
		\nyquistregion{\system} = \{ s \in \CC \bigm| \winding{\gamma}{s} = p^+ \},
	\end{equation*}
	where $p^+$ denotes the number of poles of $H(s)$ in $\CCg{0}$.
\end{definition}

Notice that $A - skbc$ is Hurwitz if and only if all the poles of $T= 1/(1+skH)$ are in $\CCl{0}$ which, by the Nyquist criterion, is equivalent to the condition $-1/{sk} \in \nyquistregion{\system}$. \change{This observation establishes a bijection between $\syncregion{\system}{k}$ and $\nyquistregion{\system}$.}

\change{\begin{lemma} \label{lem:syncnyquist}
	Let $\system$ be a minimal SISO system. For every $k \neq 0$,
	\begin{align*}
	\syncregion{\system}{k} \setminus \{0\} &= \left\{ -\frac{1}{sk} \bigm| s \in \nyquistregion{\system} \setminus \{0\} \right\}.
	\end{align*}
\end{lemma}}

\change{In view of Lemma \ref{lem:syncnyquist}, Theorem \ref{thm:synccritsta} specializes to the following SISO counterpart, which was proven in \cite{Fax:2004ua} using a different argument.}  
\begin{theorem} (Nyquist Synchronization Criterion) \label{thm:nyquistsync}
	Let $\system = (A,b,c)$ be a minimal realization of the strictly proper transfer function $H(s)$, then the network $\network = (\system, k, \graph)$ synchronizes if and only if 
	\[
	-\frac{1}{k\lambda_i} \in \nyquistregion{\system}, \quad i=1,\ldots,N-1, 
	\]
	where $\lambda_i$,  $i=1,\ldots,N-1$, are the eigenvalues of $\intmatrixred{\graph}$.  	
\end{theorem}

\change{
It is worth noticing that previous lemma and theorem directly yield the SISO counterpart of Lemma \ref{lem:stasyncable}: a collection of SISO systems $\system^N$ is output-feedback synchronizable if and only if $(\nyquistregion{\system} \cap \RR) \neq \emptyset$ when $N$ is even (or $\nyquistregion{\system} \neq \emptyset$ when $N$ is odd).
Notice that, even in the special case of SISO systems, output-feedback synchronizability is weaker than output-feedback stabilizability, as illustrated by the following example.
}

\change{
\begin{example} \label{ex:nonpip}
	Let $\system$ be a minimal realization of the transfer function
	\begin{align*}
	H(s) = \frac{(s+2)(s-0.2)}{(s+4)(s+1)(s-2)}.
	\end{align*}
	The Nyquist contour and stable region $\nyquistregion{\system}$ are illustrated in Fig. \ref{fig:nonpip}. Since $\nyquistregion{\system}$ is non-empty, any odd collection of $\system$ is OFS. But $\nyquistregion{\system}$ does not intersect the real axis, so $\system$ is not output-feedback stabilizable.
	
	As a side remark, the system $\system$ does not satisfies the parity-interlacing property\footnote{A SISO system with transfer function $H(s)$ has the parity-interlacing property if there is an even number of poles (counting multiplicities) between any pair of zeros on $[0,\infty]$.}. So, while the parity-interlacing property is necessary for output-feedback stabilizability \cite{Syrmos:1997hz,youla1974}, this example shows that it is not necessary for output-feedback synchronizability. This is not surprising since, as we have already shown, output-feedback synchronizability is a weaker condition than output-feedback stabilizability. 
	\begin{figure}[t]
		\centering
		\includegraphics[width=7cm]{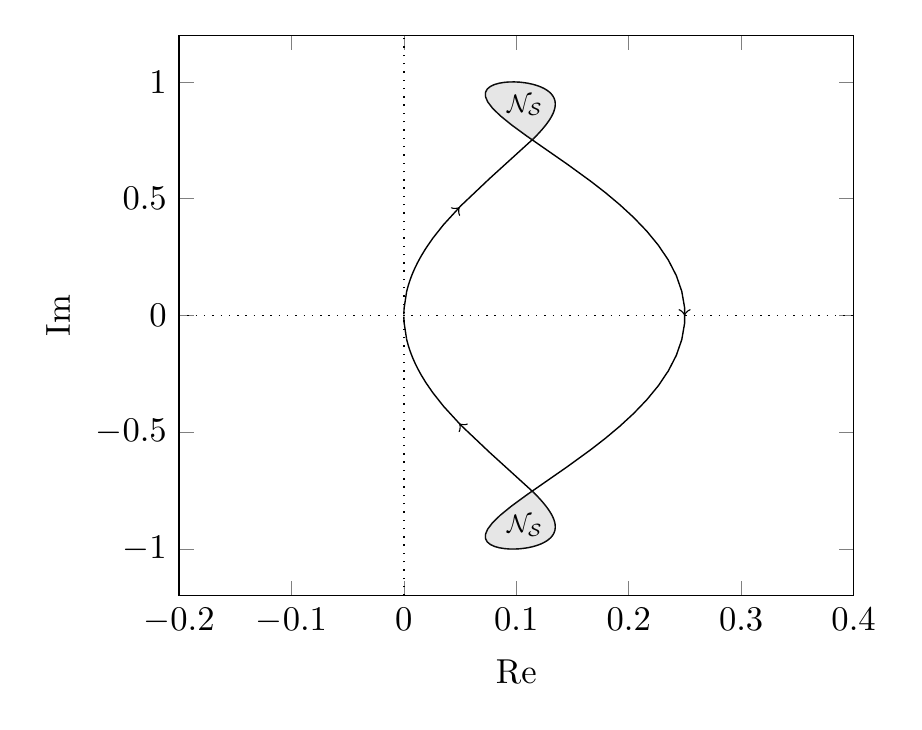}
		\caption{Nyquist contour of $\system$ in Example \ref{ex:nonpip}.}
		\label{fig:nonpip}
		\vspace{-0.4 cm}
	\end{figure}
\end{example}

The system in Example \ref{ex:nonpip} has order three. For first and second order SISO systems, output-feedback stabilizability is equivalent to output-feedback synchronizability.}
 
\begin{proposition}
\label{th:prop}
	\change{Let $\system$ be a minimal realization of a strictly proper, second order transfer function $H(s)$.} If the stable Nyquist region $\nyquistregion{\system}$ is non-empty, then $\left(\nyquistregion{\system}\cap \RR\right)  \neq \emptyset$.
\end{proposition}
\begin{proof}
	\change{Let $H(s) = N(s)/D(s)$. 
	Since $\nyquistregion{\system} \neq \emptyset$, there exists $a+ib \in \nyquistregion{\system}$ such that $P(s) = D(s) + (a+ib) N(s)$ is Hurwitz. Since $H(s)$ is a second order transfer function, $P$ has degree 2 and can be written as
	\begin{align*}
	D(s) + a N(s) + ibN(s)=&\, (s + z_1)(s + z_2) \\
	=&\,s^2 + (a_1 + a_2)s + (a_1a_2 - b_1b_2)\\ 
	&+ i\left[ (b_1+b_2)s + (a_1b_2 + a_2b_1)\right], 
	\end{align*}
	where $z_1 = a_1 + ib_1$ and $z_2 = a_2 + ib_2$.
	We therefore obtain
	\begin{align}
	bN(s) &= (b_1+b_2)s + (a_1b_2 + a_2b_1), 	\label{eq:calc1}\\
	D(s) + a N(s) &= s^2 + (a_1 + a_2)s + (a_1a_2 - b_1b_2). 	\label{eq:calc2}
	\end{align}
	Since $P$ is Hurwitz, $a_1 > 0$ and $a_2 > 0$. The following case-by-case analysis shows that $D(s) + k N(s)$ is Hurwitz for some $k \in \RR$.
	\begin{enumerate} \itemsep 0mm
		\item ($b_1b_2 \leq 0$): the coefficients of the polynomial \refe{eq:calc2} are strictly positive and therefore $D(s) + k N(s)$ is Hurwitz when $k = a$;
		\item ($b_1b_2 > 0$): the coefficients of the polynomial \refe{eq:calc1} are either all strictly positive or all strictly negative. Therefore, there exists $\bar{k} \in \RR$ such that the coefficients of the polynomial $D(s) + a N(s) + \bar{k} b N(s)$ are strictly positive. Thus $D(s) + k N(s)$ is Hurwitz when $k = a+ \bar{k}b$.
	\end{enumerate}
	By continuity, there exists $k \in \RR \setminus \{0\}$ such that $D(s) + k N(s)$ is Hurwitz, or equivalently $-1/k \in \nyquistregion{\system}$. We conclude that $\nyquistregion{\system} \cap \RR \neq \emptyset$.}
\end{proof}

\subsection{Second order systems and examples}

\change{In this section we apply our results to networks of double integrators and harmonic oscillators.}

\begin{example}[Double Integrators] 
\label{ex:siso_doubleint}
	Consider the class of double integrators described by
	\begin{equation*}
		A =
		\begin{bmatrix}
		0			& 1		\\
		0			& 0			\\
		\end{bmatrix}, \quad
		B =
		\begin{bmatrix}
		0			\\
		1			\\
		\end{bmatrix}, \quad
		C =
		\begin{bmatrix}
		c			& d \\
		\end{bmatrix}, \quad
		c \neq 0, \quad
		d \geq 0.
	\end{equation*}
The case $d\leq 0$ can be derived similarly and lead to symmetric results.   		
The transfer function is $H(s) = \left(ds+c\right)/s^2$. The system is output-feedback stabilizable if and only if $c > 0$.
By Proposition \ref{th:prop} and Lemma \ref{lem:stasyncable}, $\system^N$ is \change{OFS}   for all $N$.
\begin{figure}[t]
	\centering
	\includegraphics[width=6.cm]{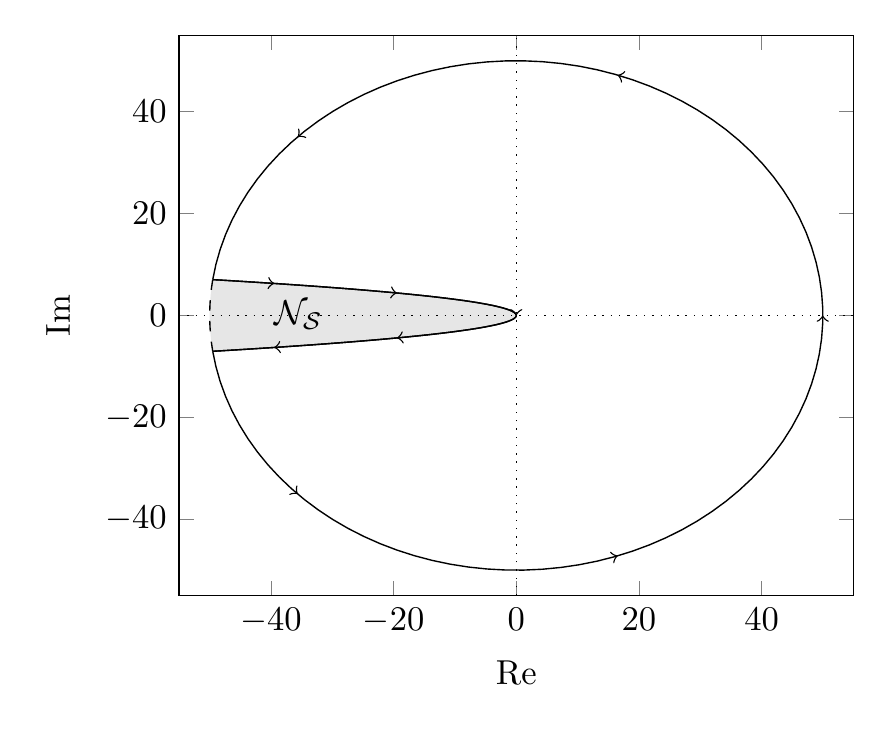} 
	\includegraphics[width=6cm]{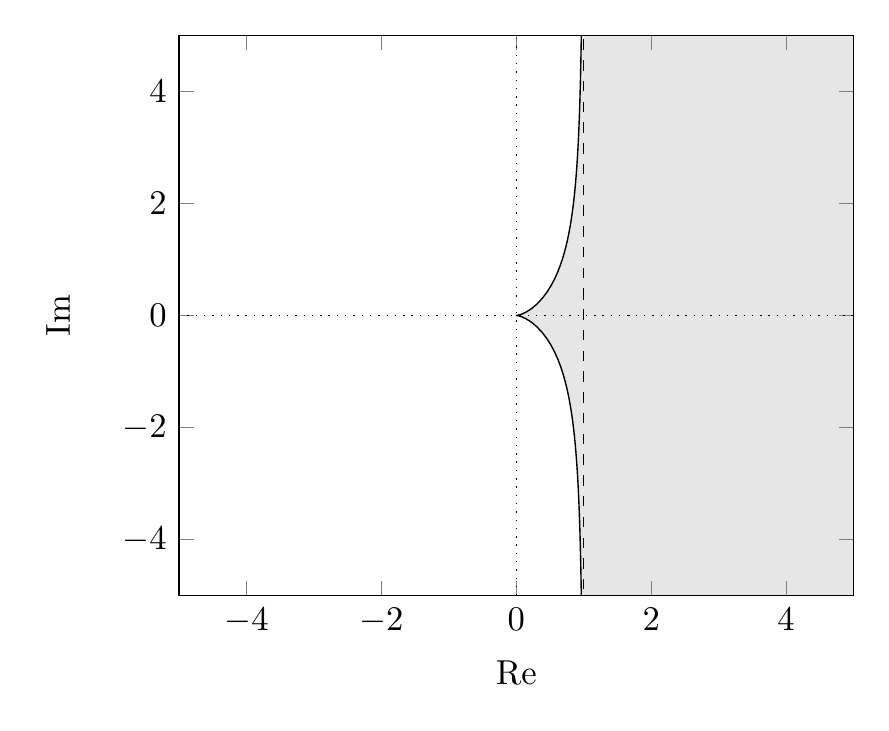} \vspace{-2mm} 
	\caption{Left: Nyquist contour of the double integrator in Example \ref{ex:siso_doubleint} for $c=d=1$. Right: Synchronization region of the double integrator in Example \ref{ex:siso_doubleint} for $c=d=1$ and $k=1$. The boundary has a vertical asymptote at $\re{s} = c/{d^2}$.}
	\label{fig:siso_doubleint}
	\vspace{-0.4 cm}
\end{figure}
The Nyquist plot and the synchronization region are shown in Fig. \ref{fig:siso_doubleint} (left). 
The existence of a vertical asymptote in the synchronization region (Fig. \ref{fig:siso_doubleint}, right) implies that, given any graph $\graph \in \graphset{N}{+}$ \change{containing a globally reachable node}, there exists a controller $k$ such that the network $(\system^{N}, k, \graph)$ synchronizes.
	
	By Theorem \ref{thm:nyquistsync}, a network $(\system^N, k, \graph)$ synchronizes if and only if $-1/(k\lambda) \in \nyquistregion{\system}$ for every $\lambda \in \spectrum{\intmatrixred{\graph}}$. From the Nyquist criterion and simple algebraic manipulations we obtain that  
	the network synchronizes if and only if  
	\begin{equation}
\frac{	\left(\im{\lambda}\right)^2}{k|\lambda|^2\re{\lambda}} < \frac{d^2}{c}, \quad k\re{\lambda}>0,
	\end{equation}
	for every $\lambda \in \spectrum{\intmatrixred{\graph}}$.
\end{example}

\begin{example}[Harmonic oscillators] \label{ex:siso_harmonic}
Consider the class of harmonic oscillators described by
	\begin{equation} \label{eq:ao}
		A =
		\begin{bmatrix}
		0			& 1		\\
		-1			& 0			\\
		\end{bmatrix}, \quad 
		B =
		\begin{bmatrix}
		0			\\
		1			\\
		\end{bmatrix}, \quad
		C =
		\begin{bmatrix}
		c			& d \\
		\end{bmatrix}, \quad
		d \geq 0.
	\end{equation}
The case $d\leq 0$ can be derived similarly and lead to symmetric results.
The transfer function is $H(s) = \left(ds+c\right)/\left(s^2+1\right)$. The system is output feedback stabilizable if and only if $d \neq 0$. The harmonic oscillator exhibits two qualitatively different Nyquist plots depending on whether $c \geq 0$ or $c < 0$. The latter corresponds to \refe{eq:ao} being non-minimum-phase.

In the case $c \geq 0$ and $d > 0$, $\nyquistregion{\system}$ contains zero as a limit point (Fig. \ref{fig:siso_harmonic}, left). The corresponding synchronization region is illustrated in Fig. \ref{fig:siso_harmonic} (right). Similar to the double integrator, $\syncregion{\system}{1}$ has a vertical asymptote. Thus, given any connected $\graph \in \graphset{N}{+}$, the network $(\system^N, k, \graph)$ synchronizes for sufficiently large $k$. From the Nyquist criterion and simple algebraic manipulations we obtain that, if $c\neq 0$, the network synchronizes if and only if 

\begin{align} \label{eq:exharmonic2}
	\frac{\left(\im{\lambda}\right)^2}{\left(\re{\lambda}\right)^2} - \frac{kd^2|\lambda|^2}{c\re{\lambda}} < \frac{d^2}{c^2}, \qquad k\re{\lambda} > 0,
\end{align}
for every $\lambda \in \spectrum{\intmatrixred{\graph}}$. If $c=0$ synchronization is obtained if and only if $k\re{\lambda} > 0$ for every $\lambda \in \spectrum{\intmatrixred{\graph}}$.  


\begin{figure}[t]
	\centering
	\includegraphics[width=6cm]{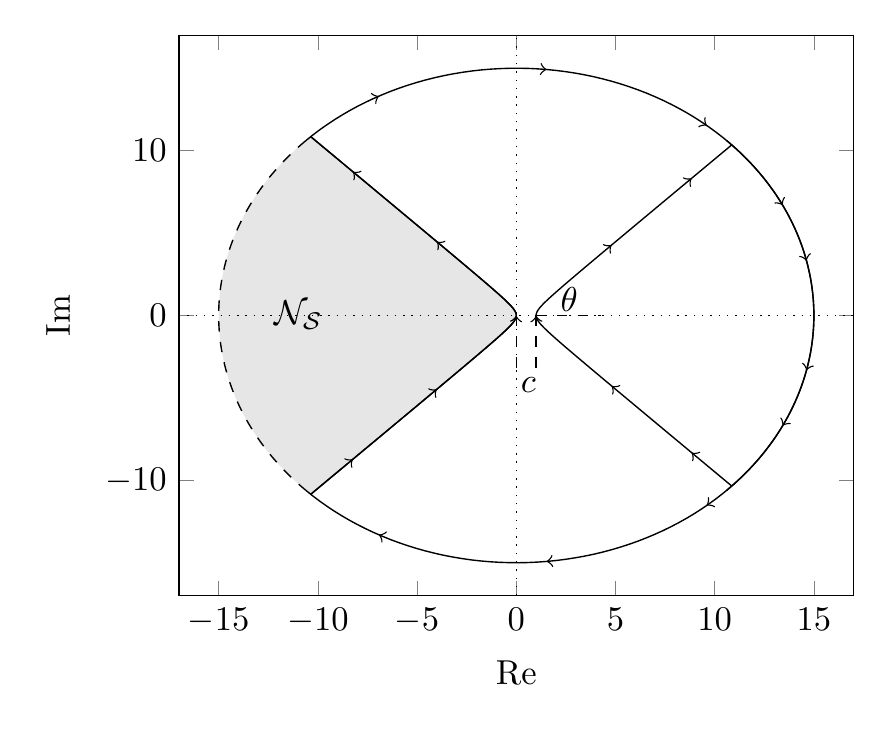}
	\includegraphics[width=6cm]{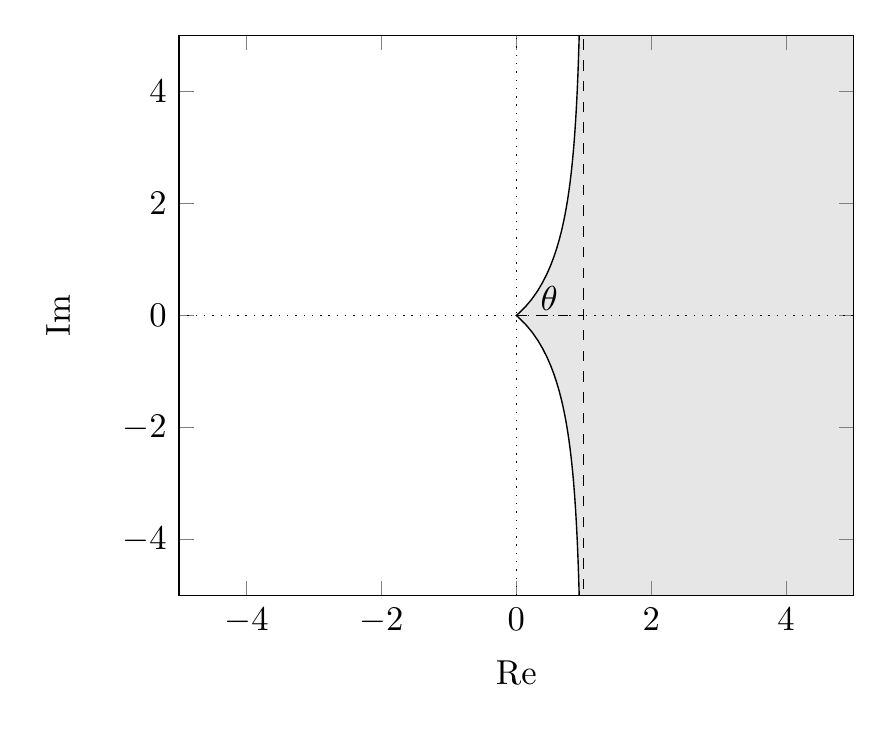}
	\caption{Left: Nyquist contour of the harmonic oscillator in Example \ref{ex:siso_harmonic} for $c=d=1$. Right: Synchronization region of the harmonic oscillator in Example \ref{ex:siso_harmonic} for $c=d=1$ and $k=1$. The vertical asymptote is located at $\re{s} = c/{d^2}$.}
	\label{fig:siso_harmonic}
	\vspace{-0.4 cm}
\end{figure}

When $c < 0$ and $d > 0$, $\nyquistregion{\system}$ is disjoint from an open neighborhood of zero (Fig. \ref{fig:siso_harmonic2}, left). Therefore, the synchronization region is a bounded subset of $\CC$ as shown in Fig. \ref{fig:siso_harmonic2} (right). Any collection $\system^N$ is OFS but the network $(\system^N, k, \graph)$ only synchronizes under weak-coupling conditions (the eigenvalues of $\intmatrixred{\graph}$ have all a small real part and $k$ is sufficiently small). 
The synchronization condition turns out to be equivalent to \refe{eq:exharmonic2}. Since $c<0$, condition \refe{eq:exharmonic2} implies that, in order for the network to synchronize, the eigenvalues of the reduced interconnection matrix must satisfy the necessary condition  
 \[
 \frac{\left(\im{\lambda}\right)^2}{\left(\re{\lambda}\right)^2}<\frac{d^2}{c^2}.
 \]

\begin{figure}
	\centering
	\includegraphics[width=6cm]{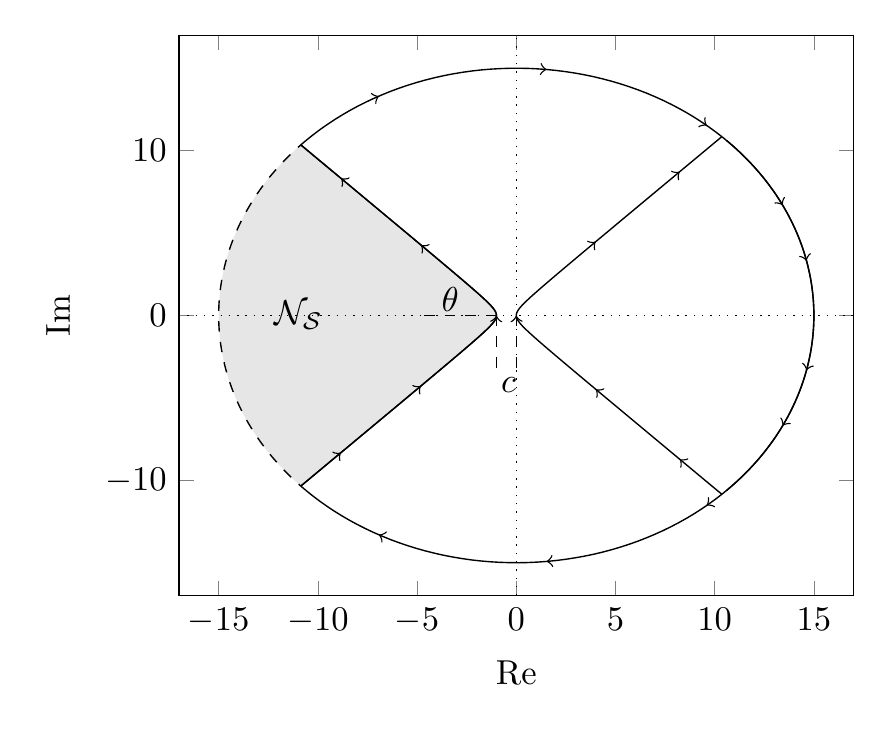}
	\includegraphics[width=6cm]{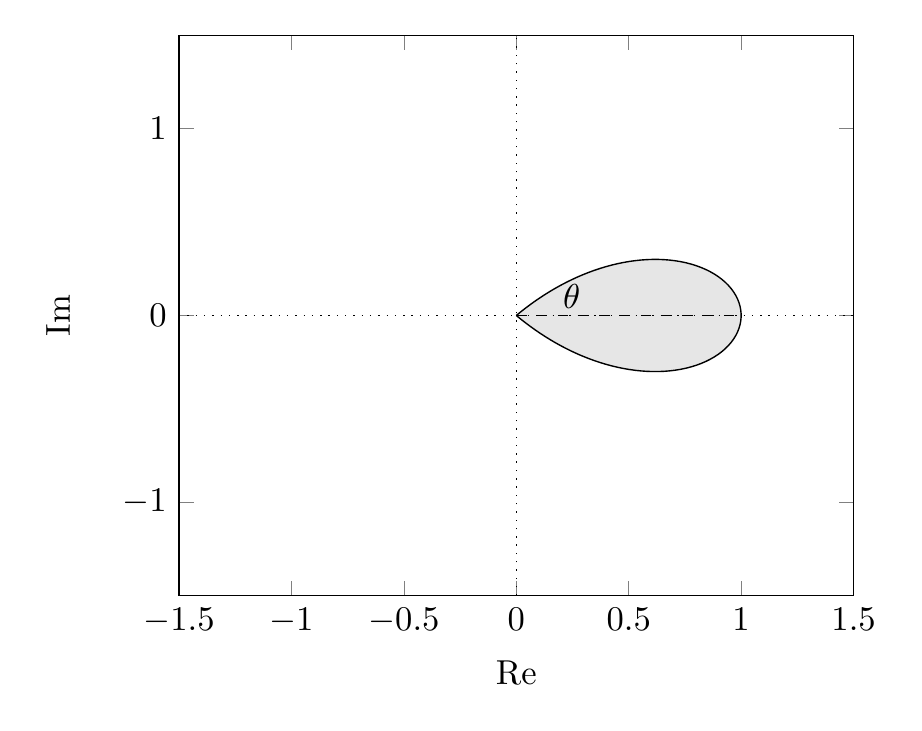}
\vspace{-2mm} 	\caption{Left: Nyquist contour of the harmonic oscillator in Example \ref{ex:siso_harmonic} for $c=-1$ and $d=1$. Right: Synchronization region of the harmonic in Example \ref{ex:siso_harmonic} for $c=-1$, $d=1$ and $k=1$. By increasing $k$ the shaded region gets smaller.}
	\label{fig:siso_harmonic2}
	\vspace{-0.4 cm}
\end{figure}
\end{example}

\vspace{-1cm}
\section{Conclusions}
\vspace{-1mm}
We addressed the problem of output-feedback synchronization for a network of LTI systems. We derived a synchronization criterion based on the notion of synchronization region and we introduced and studied \change{the notion of output-feedback synchronizability}. 
\change{In particular we have shown that a collection of output-feedback stabilizable systems is output-feedback synchronizable but output-feedback stabilizability is, in general, not necessary for output-feedback synchronizability.}
When the network is composed by SISO systems, it is shown that synchronizability is characterized by the Nyquist plot of the isolated units. 
 
In this paper we did not address the synchronization design problem, i.e., the problem of determining $K$ and $\graph$ to guarantee network synchronization. This problem remains an important direction for future work.

\bibliographystyle{IEEEtran} 
\bibliography{synlin}

\end{document}